\documentclass[10pt,a4paper,twoside]{article}
\usepackage{etex}
\usepackage{amsmath}
\usepackage{amssymb}
\usepackage{amsthm} 
\usepackage{easybmat} 
\usepackage{enumerate}
\usepackage{dsfont} 
\usepackage{epsfig} 
\usepackage{makeidx}
\usepackage[arrow, matrix, curve]{xy}
\usepackage[refpage]{nomencl} 
\usepackage{Diss}
\usepackage[OT2,T1]{fontenc}
\DeclareSymbolFont{cyrletters}{OT2}{wncyr}{m}{n}
\DeclareMathSymbol{\Sha}{\mathalpha}{cyrletters}{"58}
\begin{document}
%
%
\hyphenation{De-de-kind-ring Ga-lois-er-wei-te-rung
Trans-zen-denz-grad trans-zen-den-tes Kor-res-pon-denz Rie-mann-schen Na-mens-pa-ten Mess-lauf}
%
%
\begin{center}{\bf \LARGE The Cohen-Lenstra Heuristic:\bigskip\\ Methodology and Results\\}\vspace{2cm}
%
%
%
{\large Johannes Lengler}\vspace{1cm}\\
%
\end{center}

\thispagestyle{empty}

\begin{abstract}
In number theory, great efforts have been undertaken to study the Cohen-Lenstra probability measure on the set of all finite abelian $p$-groups. On the other hand, group theorists have studied a probability measure on the set of all partitions induced by the probability that a randomly chosen $n\times n$-matrix over $\FF_p$ is contained in a conjucagy class associated with this partitions, for $n \to \infty$.

This paper shows that both probability measures are identical. As a consequence, a multitide of results can be transferred from each theory to the other one. The paper contains a survey about the known methods to study the probability measure and about the results that have been obtained so far, from both communities.
\end{abstract}

\section{Introduction}

In 1984, Henri Cohen and Hendrik W. Lenstra published a celebrated paper \cite{CL83}, in which they conjectured that the sequence of class groups of quadratic number fields behaves essentially like a random sequence with respect to a certain probability distribution on the space of all finite abelian groups.

This probability distribution is based on the heuristic that probability for a group to occur should be inverse porportional to the number of its automorphisms. In honour of their paper, I will call this the ``Cohen-Lenstra distribution'' or ``Cohen-Lenstra probability measure''.

The consequences are immense, compared to what can be proven. So far, it is not even proven that there are infinitely many number fields with trivial class group -- a conjecture of Gau\ss\ of 1801 \cite{Gau01}. On the other hand, the Cohen Lenstra conjectures imply that for real quadratic number fields, a \emph{majority} of all these fields have trivial class group (if we neglect the $2$-part of the class group, see \cite[sect.\! 6.1]{Len09} for details).

Later on, it turned out that the Cohen-Lenstra measure occurs also in many other contexts and plays the role of a ``natural'' distribution, regulating the structure of finite abelian groups in all situations where no obvious structural obstacles for a random-like behaviour exist. The sequence of class groups of number fields is the most famous application of the Cohen-Lenstra heuristic --- not only for quadratic extensions of $\QQ$, but also much more general number field extensions are seemingly governed by similar heuristics, which may be derived from the Cohen-Lenstra heuristic. Note that apart from some special cases, all statements are conjectural but are supported by strong numerical and theoretical evidence. You may consult \cite{Mal06} or \cite{Len09} for details. Furthermore, there are completely different applications such as generating a finite abelian $p$-group ($p$ a prime) by choosing generators and imposing random relations on them with respect to some canonical Haar measure (due to Friedman and Washington in \cite{FW89}, see also \cite{Len09}).\bigskip

Due to the variety of applications and the vast consequences of the Cohen-Lenstra conjectures, number theorists have undertaken considerable efforts in order to study the Cohen-Lenstra measure in dozens of papers (\cite{CL83}, \cite{CM87}, \cite{CM90}, \cite{Len08}, \cite{Len09} and numerous others).

However, unnoticed by the number theory community, there has been another community of group theorists who encountered the Cohen-Lenstra distribution in a completely different context, namely while studying conjugacy classes of matrices. Although this theory is fully developed (e.g., cf. \cite{Ger61}, \cite{Kun81}, \cite{RS88}, \cite{Sto93}, \cite{Ful97}, \cite{Ful99}, \cite{Ful00}), the connection to the Cohen-Lenstra heuristic has slipped general attention in both direction: Neither were the group theorists aware of the Cohen-Lenstra heuristic \cite{Fulmail}, nor did the number theorists recognize the full connection to conjugacy classes (although Washington was aware of corollary \ref{cor:fixedspace} about fixed spaces \cite{Was86}, which is a special case of the general relationship).

Both communities computed important parameters and invented methods to investigate the measure. Some results were doubly obtained, but both groups may also learn new concepts from each other. The most important methods, beside direct calculations, are Cohen and Lenstra's $\zeta$-function approach, the Kung-Stong cycle index, and Fulman's two interpretations of the Cohen-Lenstra probabilities via Markov chains and via Young tableaux.\bigskip

The purpose of this paper is to give an overview of the state of the art obtained by both communities. I want to emphasize that all the results in this chapter are not my own work. My humble contribution is only to re-interpret established results in the notion of the Cohen-Lenstra heuristic.\bigskip

In this paper, I will only consider the local Cohen-Lenstra heuristic, i.e., I will only consider finite abelian $p$-groups for a fixed prime $p$. A generalization to non-primary groups is possible but requires much care. See \cite[chapter 5]{Len09} for a thorough treatment.\bigskip

The paper is structured as follows: First I give a short introduction to the Cohen-Lenstra heuristic and provide the reader with enough information to do direct calculations. Then I give a brief overview of the methods invented by several researchers. Since we unify two completely worked-out theories, space limitations will not allow us to work out all details, so I refer to the original papers for a more complete treatment. Finally, I give a collection of important quantities related to the Cohen-Lenstra measure that have been computed by those methods.

\section{Preliminaries and notation}\label{sect:prelim}

For this paper, let $p$ be a fixed prime number. We put $q:= p^{-1}$.

Throughout the paper, I will only consider finite abelian $p$-groups. For brevity, we will write ``group'' to mean ``finite abelian $p$-group'' i.e., a finite abelian group with order a power of $p$. Furthermore, we will consider groups only up to ismomorphism, so a phrase like ``sum over all groups'' really means that the sum runs over all isomorphism classes of finite abelian $p$-groups.

$\GG_p$ is the set of all (isomorphism classes of) finite abelian $p$-groups. 

For a finite set $M$, we will denote its cardinality by $\#M$.\medskip

For a finite abelian group $G$, we write $\Aut(G)$ for its automorphism group. The \emph{order $\ord(G)$} is the number of elements of $G$, the \emph{rank $\rk(G)$} is the minimal number of generators. The \emph{exponent $\exp(G)$} is the minimal integer $n>0$ such that $n\cdot G = \{0\}$. The \emph{$p$-adic order} and \emph{$p$-adic exponent} are given by the formulas

\begin{eqnarray*}
\ord_p(G) & := &  \log_p(\ord(G)),\\
\exp_p(G) & := &  \log_p(\exp(G)),
\end{eqnarray*}

respectively.\medskip

$\PPart$ is the set of all integer partitions. A partition of an integer $n\geq 0$ is a way to write $n$ as a sum of positive integers up to order of summation, e.g., 

\begin{eqnarray*}
6 & = & 6\\
&=&5+1\\
&=&4+2\\
&=&4+1+1\\
&\vdots&
\end{eqnarray*}

A partition may be uniquely described by a tuple $\underline{n} = (n_i)_{i=1,\ldots,r}$, where $r\in \NN_0$, $n_1\geq n_2 \geq \ldots \geq n_k>0$. In this representation, the $n_i$ are the different summands occurring, so $\underline{n}$ is a partition of $n = \sum_i n_i$.\smallskip

We may visualize a partition by its \emph{Young diagram}. E.g., the Young diagram of $\underline{n} = (4,2,1)$ is

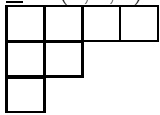
\vbox{\nboxes{4}\nmarkedboxes{0}{$\times$}}\nointerlineskip
\vbox{\nboxes{2}\nmarkedboxes{0}{$\times$}}\nointerlineskip
\vbox{\nboxes{1}\nmarkedboxes{0}{$\times$}}\nointerlineskip
\medskip

By mirroring the Young diagram of $\underline{n}$ along the main diagonal, we obtain the \emph{conjugate partition} $\underline{n}'$ of $\underline{n}$. In the above example, $\underline{n}' = (3,2,1,1)$.

\bigskip

By the Elementary Divisor Theorem, a finite abelian $p$-group can be uniquely (up to isomorphism) written in the form

$$\prod_{i=1}^k (\ZZ/p^{e_i})^{r_i},$$
where $k\in \NN_0$, $e_i,r_i \in \NN^+$ for all $i$, and where $e_1>e_2>\ldots > e_k$.

Hence, we have a canonical bijection $\GG_p \stackrel{\cong}{\rightarrow} \PPart$, and from now on we will identify both sets.

\subsection{The Cohen-Lenstra measure}

Recall that the Cohen-Lenstra measure assigns to each group a measure which is inversely proportional to the number of its automorphisms. Although we do not directly make use of it, let me give a formula for this number:

\begin{theorem}\label{thm:sizeofautomorphisms}
Let $G=\prod_{i=1}^k (\ZZ/p^{e_i})^{r_i}$ be a finite abelian $p$-group with $k\geq 0$, $e_1>\ldots>e_k>0$, $r_i >0$. The size of the automorphism group of $G$ is

$$\#\Aut(G) = \left(\prod_{i=1}^k\left(\prod_{s=1}^{r_i} (1-p^{-s}) \right)\right)\left(\prod_{1\leq i,j\leq k} p^{\min(e_i,e_j)r_i r_j}\right).$$
\end{theorem}

\begin{proof}
\cite[theorem 1.2.10]{Len09}
\end{proof}

Now we turn to the definition of the Cohen-Lenstra weight and the Cohen-Lenstra measure:

\begin{definition}
The \emph{Cohen-Lenstra weight} $w$ is the measure on the set $\GG_p$ of all finite abelian $p$-groups that is defined via 
$$w(\{G\}) := \frac{1}{\#\Aut(G)}\qquad \text{for all one-element sets $\{G\}\subset \GG_p$}.$$

The \emph{Cohen-Lenstra (probability) measure} $P$ is the probability measure on $\GG_p$ that is obtained by scaling $w$:
$$P(M) := \frac{w(M)}{w(\GG_p)}\qquad\text{for $M\subseteq \GG_p$}.$$

In slight abuse of notation we will write $w(G)$ and $P(G)$ instead of $w(\{G\})$ and $P(\{G\})$, respectively, when we measure one-element sets $\{G\}\subset \GG_p$. 

\end{definition}

The above definition of the Cohen-Lenstra measure makes only sense if $w(\GG_p)$ is finite. Fortunately, this is the case. Hall \cite{Hal38} has shown that 

\begin{theorem}\label{thm:weightsum}
The Cohen-Lenstra weight of the set of all finite abelian $p$-groups is

$$w(\GG_p) = \prod_{i=1}^{\infty} (1-p^{-i})^{-1} < \infty.$$
\end{theorem}
 
\newpage

\section{Elementary calculations}

The explicit formulas (theorems \ref{thm:sizeofautomorphisms} and \ref{thm:weightsum}) enable us to compute some values rather easily. For example, given a group $G\in\GG_p$, we are given an explicit formula for $P(G)$. As a special case, let me give the probability that a $p$-group is the trivial group $0$. Since $w(0) = 1$, we obtain
 
$$P(0) = \prod_{i=1}^{\infty} (1-p^{-i}).$$

Using $q$-series identities, we may compute some other probabilities. For example, the probability that a random group is cyclic (i.e., has rank $\leq 1$), is (with $q=\frac{1}{p}$, as usual)

\begin{eqnarray*}
P(G \text{ cyclic}) & = & \frac{1}{w(\GG_p)}\sum_{G\text{ cyclic}} w(G)\\
& = & \left(\prod_{i=1}^{\infty} (1-q^i)\right)\sum_{e=0}^{\infty}\frac{q^e}{1-q}\\
& = & \left(\prod_{i=1}^{\infty} (1-q^i)\right)\frac{1}{(1-q)^2}\\
& = & \frac{1}{1-q}\prod_{i=2}^{\infty} (1-q^i)\\
& = & \frac{p}{p-1}\prod_{i=2}^{\infty} (1-p^{-i}).
\end{eqnarray*}

The calculation was pleasantly simple. Now let us compare this to what happens if we try to treat the slightly more complicated question of how likely it is for a random group to have rank $2$. Within the computation we distinguish two different cases, corresponding to the possible group structures $G=(\ZZ/p^e)^2$, and $G=\ZZ/p^{e_1} \times \ZZ/p^{e_2}$, $e_1 > e_2$:

\begin{eqnarray*}
P(\rk(G) = 2) \!\! & = & \!\!\frac{1}{w(\GG_p)}\sum_{\rk(G) = 2} w(G)\\
& = & \!\!\left(\prod_{i=1}^{\infty} (1-q^i)\!\right)\!\!\left(\sum_{e=1}^{\infty}\frac{q^{4e}}{(1-q)(1-q^2)} +\!\! \sum_{e_2=1}^{\infty}\sum_{e_1=e_2+1}^{\infty}\frac{q^{e_1+3e_2}}{(1-q)^2}\right)\\
& = & \!\!\left(\prod_{i=1}^{\infty} (1-q^i)\!\right)\!\!\left(\frac{q^{4}}{(1-q)(1-q^2)(1-q^4)} +\right.\\
&&\phantom{\left(\prod_{i=1}^{\infty} (1-q^i)\right)}\left.+ \frac{1}{(1-q)^2}\sum_{e_2=1}^{\infty}q^{3e_2}q^{e_2+1}\frac{1}{(1-q)}\right)\\
& = & \!\!\left(\prod_{i=1}^{\infty} (1-q^i)\!\right)\!\!\left(\!\frac{q^{4}}{(1-q)(1-q^2)(1-q^4)}\! +\! \frac{q^5}{(1-q)^3(1-q^4)}\right)\\
& = & \!\!\left(\prod_{i=1}^{\infty} (1-q^i)\!\right) \!\frac{q^4-q^5+q^5+q^6}{(1-q)^2(1-q^2)(1-q^4)}\\
& = & \!\!\left(\prod_{i=1}^{\infty} (1-q^i)\!\right) \!\frac{q^4}{(1-q)^2(1-q^2)^2}.
\end{eqnarray*}

Recalling that this was still one of the ``easier'' cases, we see that this approach soon becomes quite cumbersome. It \emph{is} possible to get general results about order and rank of a random group in this way (Bernd Mehnert will present some of these calculations in his PhD-thesis \cite{Meh}), but this requires a highly skillful handling of $q$-series identities, which we do not want to expect from the user. 

So we need other tools to enhance our ability to compute interesting values. The next sections will provide such tools.

\section{Zeta functions}

Cohen and Lenstra embed what I call the Cohen-Lenstra weight $w$ into a larger family of measures $w_k$ as follows. For a finite abelian $p$-group $G$, let $s_k(G)$ be the number of surjective homomorphisms $\ZZ^k\to G$ (or, equivalently, $\ZZ_p^k \to G$). Then they define 

$$w_k(G) := \frac{s_k(G)}{|G|^k}w(G).$$

Note that the denominator equals the number of \emph{all} (not necessarily surjective) homomorphisms $\ZZ^k\to G$.

Then we may compute $w_k(G)$ as
 
\begin{equation}\label{eq:twistedprob}w_k(G) = \begin{cases} \displaystyle w(G)\prod_{i=k-r+1}^k(1-q^i) & \text{ if } k\geq r:=\rk(G),\vspace{1ex} \\ \displaystyle \ 0 & \text{ otherwise.}\end{cases}\end{equation}

(\cite[Prop.\! 3.1]{CL83}).

In particular, we may recover $w(G)$ as

$$w(G) = \lim_{k\to\infty} w_k(G).$$

Now we define the $k$-$\zeta$-function over $\GG_p$ as

$$\zeta_k^{(p)}(s) := \sum_{G\in\GG_p}\frac{w_k(G)}{|G|^s}.$$

Then $\zeta_k^{(p)}$ converges for $\Re(s) > -1$ and may be computed explicitly by

$$\zeta_k^{(p)}(s) = \prod_{i=1}^{k}\frac{1}{(1-p^{-s-i})}$$

(\cite[Cor.\! 3.7]{CL83}).

In particular, this implies the formula $\zeta_{k_1+k_2}^{(p)}(s) = \zeta_{k_1}^{(p)}(s+k_2)\zeta_{k_2}^{(p)}(s)$.

We need one last definition: Let $f: \GG_p \to \CC$ be an integrable function. We define

$$\zeta_k^{(p)}(f;s) := \sum_{G\in\GG_p}\frac{w_k(G)f(G)}{|G|^s}.$$

Then the expected value $E(f)$ of $f$ may be computed as 

$$E(f) = \lim_{k\to\infty}\frac{\zeta_k^{(p)}(f;0)}{\zeta_k^{(p)}(0)}.$$

(This is an analogue of \cite[Cor.\! 5.5]{CL83}, only for local groups.)\medskip

Often, it is easier to compute the $\zeta$-function of $f$ than to compute the expected value of $f$ directly. In this way, Cohen and Lenstra compute explicit formulas for the rank and the order of groups, and for some other functions (cf. the discussion in section \ref{numresults}).

Their approach has two more advantages. Firstly, we get almost for free a treatment of the twisted probability measure $P_u$ discussed in section \ref{sect:u-probs}, which is of special interest for number field extensions that are not imaginary quadratic (see \cite{Mal06} or \cite[chap.\! 6]{Len09} for details).

More precisely, we may compute the expected value $E_u(f)$ of $f$ with respect to the twisted probability measure $P_u$ as

$$E_u(f) = \lim_{k\to\infty} \frac{\zeta_{k}^{(p)}(f;u)}{\zeta_{k}^{(p)}(u)}$$

(\cite[Cor. 5.5]{CL83}).

The second advantage is that the approach gives a way to obtain some statements about the global setting. We may analogously define a $\zeta$-function over the global set $\GG$, it only has a smaller domain of convergence. More precisely, it converges for $\Re(s) > 0$ and has a simple pole in $0$. Therefore, under some technical conditions the expected value of certain \emph{global} functions $f:\GG\to\CC$ may be computed as 

$$E(f) = \lim_{s\to 0} \lim_{k\to\infty}\frac{\zeta_k(f;s)}{\zeta_k(s)}$$

(\cite[Thm. 5.5]{CL83}), and we only need to compute the residues of the global $\zeta$-functions. However, note that we cannot use this approach to define a probability measure on $\GG$. Taking the sets for which the above limit exists only yields a content (i.e., a ``measure'' that is only finitely additive). For a thorough discussion, see \cite[chap.\! 5]{Len09}.

%
%
%
%
%
%

\section{The Cohen-Lenstra heuristic: Interpretation via conjugacy classes}\label{sect:conjugacyclasses}

Recall that $p$ is a fixed prime number.

Consider the general linear group $\GL(n, p)$\nomenclature[Glnp]{$\GL(n,p)$}{general linear group over $\FF_p$}
 of invertible $n\times n$-matrices over $\FF_p$. Then each conjugacy class can be represented by a matrix in \emph{Jordan-Chevalley normal form}\index{Jordan-Chevalley normal form}.
 
Before I describe this form, let me define the \emph{companion matrix}\index{companion matrix} $C(\varphi)$ of a normalized polynomial $\varphi = X^m+a_{m-1}X^{m-1}+\ldots+a_1X+a_0$. We set $C(\varphi)$ to be the $m\times m$-matrix 

$$ C(\varphi) := 
\left(\begin{matrix}
0 & 1 & 0 & \ldots & 0 \\
0 & 0 & 1 & \ldots & 0 \\
\vdots & \vdots & \vdots & \ddots & \vdots \\
0 & 0 & 0 & \ldots & 1 \\
-a_0 & -a_1 & -a_2 & \ldots & -a_{m-1}
\end{matrix}\right).
$$

Now back to the normal form. It looks as follows: For every monic irreducible polynomial $\phi$ of degree $m$ over $\FF_p$ and every positive integer $s$ we may have an arbitrary number (possibly $0$) of ($\phi,s$)-Jordan blocks. Each Jordan block is a square of size $sm$ and is the companion matrix of the polynomial $\phi^s$. The normal form then has the form

$$ 
\left(\begin{matrix}
J_1 & 0 & 0 & \ldots & 0 \\
0 & J_2 & 0 & \ldots & 0 \\
0 & 0 & J_3 & \ldots & 0 \\
\vdots & \vdots & \vdots & \ddots & \vdots \\
0 & 0 & 0 & \ldots & J_r \\
\end{matrix}\right),
$$

where $J_k$ runs through all the Jordan blocks. We only require that the sizes of the Jordan blocks add up to $n$.

The normal form works over every field. In section \ref{sect:cycleindex}, we will also work over the field $\FF_{p^i}$, but for the basic theorems it suffices to consider $\FF_p$. Note that over an algebraically closed field (such as $\CC$) all irreducible polynomials are linear and the Jordan-Chevalley normal form reduces to a slight variation of the ordinary Jordan normal form.

In order to specify a normal form we must specify for every monic irreducible polynomial $\phi$ and any $s > 0$ how many ($\phi,s$)-Jordan blocks occur. In other words, for each $\phi$ we must specify a partition. We call this partition $\underline{\lambda}_{\phi}$. For example, if we have $2$ blocks of size $3m$ and $3$ blocks of size $m$ then this corresponds to the partition $(3,3,1,1,1)$. In order for the matrix to be invertible we must require that $\underline{\lambda}_{X} = ()$.

On the other hand, every collection of partitions $(\underline{\lambda}_{\phi})_{\phi}$ with the properties

\begin{itemize}
	\item $\underline{\lambda}_{X} = \underline{0}$ and
	\item $\sum_{\phi,s} (\deg{\phi})\lambda_{\phi,s} = n$
\end{itemize}

defines a (unique) conjugacy class in $\GL(n,p)$.\medskip

From now on, we fix a monic polynomial $\phi \neq X$ over $\FF_p$ of degree $1$.

Let $\underline{\lambda}$ be a partition. Pick a random matrix in $\GL(n,p)$ uniformly at random. Then we get a certain probability for the event $\underline{\lambda}_{\phi} = \underline{\lambda}$. 
 
Fulman proved the following theorem.

\begin{theorem}\label{thm:conjugacy}
Let $\phi$ be any monic polynomial over $\FF_p$ of degree $1$ and let $\underline{\lambda}$ be a partition. As $n\rightarrow \infty$, the probability (in the sense above) that $\underline{\lambda}_{\phi} = \underline{\lambda}$ for a random matrix in $\GL(n,p)$ (chosen uniformly at random) converges to the CL-probability $P(\underline{\lambda})$.
\end{theorem}

\begin{proof}
\cite[Sect. 3.3, Cor. 5 and Sect. 2.7, Lemma 6 and Thm. 5 with $u=1$ and $N\rightarrow \infty$]{Ful97}.
\end{proof}

\begin{remark}~
\begin{itemize}
\item Fulman uses in his thesis a slightly different way of taking the $n\rightarrow \infty$ limit. Rather, he chooses a parameter $0<u<1$, then picks the integer $n$ with probability $(1-u)u^n$ and chooses a random matrix from $\GL(n,p)$ (cf. \cite[p.557f.]{Ful99}). Then he proceeds as above. However, it is easy to see that letting $u \rightarrow 1$ in this setting yields the same limit as letting $n\rightarrow \infty$ in the theorem above. We only need to interchange two limits, but this is no problem since all statements concern formal power series identities with positive convergence radius.

The reason why Fulman chose the parameter $u$ instead of $n$ will become clear in section \ref{sect:cycleindex} about the cycle index.
\item Fulman studies also the probability distribution for monic polynomials $\phi$ of higher degree. This yields similar distributions with similar formulas, only it does not give exactly the Cohen-Lenstra probability. We will encounter these other distributions in the context of the Kung-Stong cycle index in section \ref{sect:cycleindex}.
\end{itemize}
\end{remark}

The theorem allows us to transfer a multitude of methods and results from a whole community of researchers to the Cohen-Lenstra heuristic. I start with reviewing a very interesting interpretation of the Cohen-Lenstra heuristic in terms of Markov chains due to Fulman.

\section{Interpretation via Markov chains}\label{Markovchains}

In his PhD thesis, Fulman gave two interpretations of the Cohen-Lenstra probability. One as the outcome of a probabilistic algorithm, one as the weight in the Young lattice with certain transition probabilities. I review both interpretations in the setting that is relevant to us.

First I present what Fulman calls the ``Young Tableau Algorithm'' (\cite{Ful99}). Recall that $p$ is a fixed prime.\pagebreak[2]

\begin{algorithm}\index{Markov chains algorithm}~

\begin{enumerate}
	\item[0.] Start with $\underline{\lambda}$ the empty partition. Also start with $N = 1$ and with a collection of coins indexed by the natural numbers, such that coin $i$ has probability $\frac{1}{p^i}$ of heads and $1- \frac{1}{p^i}$ of tails.
	\item[1.] Flip coin $N$. If the outcome is tails then set $N:= N+1$ and redo step 1, otherwise go to step 2.
	\item[2.] Choose an integer $S > 0$ according to the following rule. Set $S:=1$ with probability $\frac{p^{N-\underline{\lambda}_1}-1}{p^N-1}$. For $s>1$, set $S:=s$ with probability $\frac{p^{N-\underline{\lambda}_s}-p^{N-\underline{\lambda}_{s-1}}}{p^N-1}$. Then increase $\underline{\lambda}_S$ by $1$ and go to step 1.
\end{enumerate}

In step 2, we use the convention that all undefined entries of $\underline{\lambda}$ are $0$. In particular, if we increase some $\underline{\lambda}_s$ that is not defined then after increasing the entry is $1$.

The algorithm does not halt, but $\underline{\lambda}$ converges against some limit partition $\underline{\lambda}_{\infty}$ (cf. theorem \ref{thm:Youngtableutermination} below). The ouput of the algorithm is the \emph{conjugate} partition $\underline{\lambda}_{\infty}'$ of $\underline{\lambda}_{\infty}$.

\end{algorithm}

\begin{example}

Assume that we are at step 1 with $\underline{\lambda} = (3,2,1,1)$, so the Young diagram of $\underline{\lambda}$ is

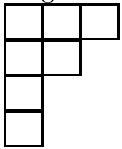
\vbox{\nboxes{3}\nmarkedboxes{0}{$\times$}}\nointerlineskip
\vbox{\nboxes{2}\nmarkedboxes{0}{$\times$}}\nointerlineskip
\vbox{\nboxes{1}\nmarkedboxes{0}{$\times$}}\nointerlineskip
\vbox{\nboxes{1}\nmarkedboxes{0}{$\times$}}\nointerlineskip
\medskip

Assume further that $N=4$ and that coin $4$ comes up heads, so we go to step $2$. We add to $\underline{\lambda}_1$ with probability $\frac{p-1}{p^4-1}$, to $\underline{\lambda}_2$ with probability $\frac{p^2-p}{p^4-1}$, to $\underline{\lambda}_3$ with probability $\frac{p^3-p^2}{p^4-1}$, to $\underline{\lambda}_4$ with probability $0$, and to $\underline{\lambda}_5$ with probability $\frac{p^4-p^3}{p^4-1}$.

Assume that we choose $S=1$ and increase $\underline{\lambda}_1$, thus getting $\underline{\lambda} = (4,2,1,1)$ with Young diagram

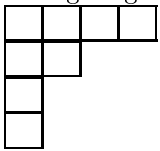
\vbox{\nboxes{4}\nmarkedboxes{0}{$\times$}}\nointerlineskip
\vbox{\nboxes{2}\nmarkedboxes{0}{$\times$}}\nointerlineskip
\vbox{\nboxes{1}\nmarkedboxes{0}{$\times$}}\nointerlineskip
\vbox{\nboxes{1}\nmarkedboxes{0}{$\times$}}\nointerlineskip
\medskip

We return to step $1$ and still have $N=4$. Assume that again coin $4$ comes up heads and we go to step 2. Now we add to $\underline{\lambda}_1$ with probability $0$, to $\underline{\lambda}_2$ with probability $\frac{p^2-1}{p^4-1}$, to $\underline{\lambda}_3$ with probability $\frac{p^3-p^2}{p^4-1}$, to $\underline{\lambda}_4$ with probability $0$, and to $\underline{\lambda}_5$ with probability $\frac{p^4-p^3}{p^4-1}$. Then we return to step 1.

\end{example}

\begin{remark}
The name ``Young Tableau Algorithm'' refers to the concepts of Young tableaux. A Young tableau is a Young diagram where the boxes are labelled with $1,\ldots,n$ ($n$ the size of the Young diagram). The labels must be given in a way that for any $1\leq i \leq n$ the boxes $1,\ldots,i$ form again a Young diagram. You may think of a Young tableau as a Young diagram together with an ordering which tells you how to build up the diagram from scratch. Since the algorithm does exactly this (building up Young diagrams block by block), the name is appropriate.
\end{remark}

\begin{theorem}\label{thm:Youngtableutermination}
With probability $1$, the algorithm outputs a \emph{finite} partition. For any given partition $\underline{\lambda}$, the probability that the algorithm outputs $\underline{\lambda}$ equals the Cohen-Lenstra probability $P(\underline{\lambda})$. 
\end{theorem}

\begin{proof}[Proof of theorem \ref{thm:Youngtableutermination}]~

\cite[Thm. 1]{Ful99} with $u=1$ and $q=p$. The author states termination of the algorithm only for the case $u<1$, but his proof implies termination for $u=1$ as well.
\end{proof}

Since the concept of such an algorithm may be unfamiliar to the reader, let me rephrase the finiteness statement of the theorem. Let us say the algorithm has been running for some (finite) time and is in some state $\underline{\lambda}$. Then there is a \emph{positive} probability that the algorithm will not add any more blocks to $\underline{\lambda}$ in all the (infinitely many) forthcoming steps of the algorithm. Thus, there is a positive probability that the algorithm outputs $\underline{\lambda}$. On the other hand, the probability that the algorithm adds infinitely many blocks to $\underline{\lambda}$ in the (infinite) sequel of the algorithm is $0$. Hence, with probability $1$ the algorithm outputs a finite partition.

\begin{remark}

It may be of interest to state one intermediate result in Fulman's proof. Namely, the probability $P_{alg}^N(\underline{\lambda})$ that the generic partition of the algorithm equals $\underline{\lambda}$ at the time when coin $N$ comes up tails is

\begin{equation}\label{eq:algFulman}P_{alg}^N(\underline{\lambda}) = \begin{cases} 
\displaystyle\left(\prod_{i=N-\underline{\lambda}_1+1}^{N}(1-p^{-i})\right)\left(\prod_{i=1}^N(1-p^{-i})\right)w(\underline{\lambda}) & \text{ if } \underline{\lambda}_1 \leq N,\vspace{1ex} \\
\displaystyle \ 0 & \text{ if } \underline{\lambda}_1 > N,
\end{cases}\end{equation}

where $w(\underline{\lambda})$ is the Cohen-Lenstra weight of $\underline{\lambda}$.

Evidently, this converges to $P(\underline{\lambda})$ as $N \rightarrow \infty$.

Formula \eqref{eq:algFulman} is of particular interest because it also occurs in a different context in a paper \cite{FW89} of Friedman of Washington. More precisely, the probability that  $\underline{\lambda}$ is the intermediary result in Fulman's algorithm when coin $N$ comes up tails equals the probability that a random matrix $A\in \ZZ_p^{n\times n}$ (with respect to the Haar measure) has cokernel $\underline{\lambda} \in \PPart \cong \GG_p$.

So the algorithm is compatible with the graded (by $n$) structure of the process of choosing $n$ generators and $n$ relations described in \cite{FW89} and \cite[sect.\! 2.2.3]{Len09}.
\end{remark}

\section{Interpretation in the Young lattice}

Fulman's second interpretation is perhaps even more interesting from our point of view, since it connects more directly to the \emph{CL-weight} rather than to the \emph{CL-probability}.

This approach makes use of the \emph{Young lattice}. The Young lattice is a directed graph with vertex set $\GG_{\PPart}$ ($=\GG_p$, but independent of $p$!). There is a directed edge from $\underline{\lambda}$ to $\underline{\mu}$ if and only if the Young diagram of $\underline{\lambda}$ is contained in the Young diagram of $\underline{\mu}$ and $\size(\underline{\lambda}) = \size(\underline{\mu}) -1$.\index{Young lattice}

For the algorithm we will index the vertices by the conjugate $\underline{\lambda}'$ of $\underline{\lambda}$. This does not affect the edge set. Note that there is a directed edge from $\underline{\lambda}$ to $\underline{\mu}$ if and only if there is an index $i_0$ such that $\underline{\mu}'_{i_0} = \underline{\lambda}'_{i_0} +1$ and $\underline{\mu}'_{i} = \underline{\lambda}'_{i}$ for all $i \neq i_0$.

\begin{theorem}
Put weights $m_{\underline{\lambda}',\underline{\mu}'}$ on the edges in the Young lattice as follows:

\begin{enumerate}
	\item $$m_{\underline{\lambda}',\underline{\mu}'} = \frac{1}{p^{\underline{\lambda}'_1}(p^{\underline{\lambda}'_{1}+1}-1)} \qquad \text{if } \underline{\mu}'_1 = \underline{\lambda}'_1 +1.$$
	\item $$m_{\underline{\lambda}',\underline{\mu}'} = \frac{p^{-\underline{\lambda}'_{s}}-p^{-\underline{\lambda}'_{s-1}}}{p^{\underline{\lambda}'_{1}}-1} \qquad \text{if } \underline{\mu}'_s = \underline{\lambda}'_s +1 \text{ for $s>1$.}$$
\end{enumerate}

Then the following formula holds for the Cohen-Lenstra weight $w$ and for any $\underline{\lambda}\in \GG_{\PPart}$ of size $\lambda$:

$$w(\underline{\lambda}) = \sum_{\gamma'}\prod_{i=0}^{\lambda -1} m_{\gamma'_i,\gamma'_{i+1}},$$

where $\gamma' = (\gamma'_1,\ldots,\gamma'_{\lambda})$ runs over all directed paths from the empty partition to $\underline{\lambda}'$ in the Young lattice.
\end{theorem}

\begin{proof}
\cite[Thm. 2]{Ful99}
\end{proof}

\begin{remark}
A brief calculation shows that for any partition $\underline{\lambda}\in \PPart$ the sum of the weights of edges out of $\underline{\lambda}\neq ()$ is $\frac{p}{p^{\underline{\lambda}'_{1}+1}-1} < 1$. (For $\lambda = ()$, it is $\frac{1}{p-1} < 1$.) Therefore, the edge weights can also be viewed as transition probabilities, provided that we allow for halting.
\end{remark}

\section{The Kung-Stong cycle index}\label{sect:cycleindex}

This is a powerful tool for investigating conjugacy classes of groups, deve\-loped by Kung, Stong and Fulman. The techniques apply also to more general algebraic groups, but for us only the group $\GL(n,p)$ is of interest. Recall (section \ref{sect:conjugacyclasses}) that a conjugacy class of a matrix $M \in \GL(n,p)$ is described by assigning a partition $\underline{\lambda}_{\phi}(M)$ to each monic irreducible polynomial $\phi \neq X$ such that $\sum_{\phi,s} (\deg{\phi})\underline{\lambda}_{\phi,s}(M) = n$. 

\begin{definition}\index{cycle index}\nomenclature[ZGL]{$Z_{\GL(n,p)}$}{cycle index of $\GL(n,p)$}
For all $\phi \neq X$ and all partitions $\underline{\lambda}$, let $x_{\phi,\underline{\lambda}}$ be a variable. Then the \emph{cycle index $Z_{\GL(n,p)}$} is defined as follows:

$$Z_{\GL(n,p)} := \frac{1}{|\GL(n,p)|}\sum_{M \in \GL(n,p)}\prod_{\phi \neq X}x_{\phi,\underline{\lambda} _{\phi}(M)}.$$
\end{definition}

This cycle index is connected with the Cohen-Lenstra probability. In order to formulate the connection, we embed the CL-probability in a larger class of probability measures on $\GG_p$. For any power $p^i$ of $p$ and real number $0<u<1$, we define a probability distribution $P_{u,p^i}$ on $\GG_p$ as follows. Fix a monic polynomial $\phi \neq X$ over $\FF_{p^i}$ of degree $1$. Choose an integer $n$ randomly according to the probability distribution $k \mapsto (1-u)u^k$. Now pick a matrix $M \in \GL(n,p^i)$ uniformly at random. Then the pair $(M,\phi)$ defines a partition $\underline{\lambda}_{\phi}(M)$. We define $P_{u,p^i}(\underline{\lambda})$ to be the probability that $\underline{\lambda}_{\phi}(M) = \underline{\lambda}$. (This is easily seen to be independent of the choice of $\phi$.)

Recall that the CL-probability is obtained from $P_{u,p^i}$ by setting $i:= 1$ and letting $u\rightarrow 1$.

Explicit formulas for $P_{u,p^i}$ are given in \cite[sect.\!\! 2]{Ful99}. (The author writes $M_{(u,q)}$ instead of $P_{u,p^i}$.) 

Now we can state the following theorem due to Kung \cite{Kun81} and Stong \cite{Sto88}:

\begin{theorem}
$$(1-u)\left(1+\sum_{n=1}^{\infty}Z_{\GL(n,p)}u^n\right) = \prod_{\phi \neq X} \sum_{\underline{\lambda}}x_{\phi,\underline{\lambda}}P_{u,p^{\deg(\phi)}}(\underline{\lambda}).$$
\end{theorem}

\begin{proof}
\cite[Thm. 10]{Ful97}
\end{proof}

We will not go into too much detail about the techniques that extract interesting consequences from this formula, but the essential point is -- possibly after some formula manipulation -- comparing the coefficients of $u^n$ on both sides. I refer to \cite{Ful97}, \cite{Ful99} and \cite{Ful00} for tons of examples.

\section{A collection of results}\label{numresults}

In this section I cite results that were obtained by the number theory community and the group theoretic community. Some of them were found by both communities, some not.

Recall that a ``randomly chosen group'' really means a randomly chosen finite abelian $p$-group with respect to the Cohen-Lenstra probability with $q = \frac{1}{p}$ regarded as a formal variable.

\subsection{Order}

\begin{theorem}
The probability that a randomly chosen group has order $p^n$ is 

$$P(\ord(G) = p^n) = q^n\prod_{i=n+1}^{\infty}(1-q^i).$$
\end{theorem}

\begin{proof}\cite[Cor. 3.8]{CL83}\end{proof}

\subsubsection{Higher moments of the order}

Recall that the $k$-th moment of a random variable $X$ is the expected value of $X^k$.\index{moments, higher}

The higher moments of the order of a random group do not exist if $k\geq 1$. (I.e., their values are $\infty$.) However, for the \emph{local order} (section \ref{sect:prelim}) we obtain something meaningful. In his PhD thesis \cite{Meh}, yet to appear, Bernd Mehnert gives a stunning description in terms of Eisenstein series:

For $k\geq 1$ let 
$$E_k(q) := \sum_{n=1}^{\infty}\sigma_{k-1}(n)q^n$$

be the $k$-th \emph{Eisenstein series} deprived of its constant term, where $\sigma_{i}(n) = \sum_{1 \leq d|n} d^i$ is the $i$-th divisor sum. Note that we have defined the Eisenstein series both for odd and even $k$.

For a group $G = \prod_{i=1}^l (\ZZ/p^{e_i})^{r_i}$ in standard form (in particular, all $e_i$ are mutually distinct) of order $p^k$, let

$$f_{G}(X_1,\ldots,X_k) := k!\prod_{i=1}^{l}\frac{X_{e_i}^{r_i}}{r_i! (e_i!)^{r_i}}$$

and

$$f_k(X_1,\ldots,X_k) := \sum_{G \text{ group of order $p^k$}}f_{G}(X_1,\ldots,X_k).$$

\begin{theorem}
With the above notation, the $k$-th moment $M_k$ of the local order of a random $p$-group is

$$\sum_{n\geq 0} n^k \cdot P(\ord_{p}(G)=n) = f_k(E_1,E_2,\ldots,E_k).$$
\end{theorem}
\begin{proof}
\cite{Meh}.
\end{proof}

For example, $M_1 = E_1$, $M_2 = E_1^2+E_2$, $M_3 = E_1^3+3E_1E_2+E_3$, $M_4 = E_1^4+6E_1^2E_2+3E_2^2+4E_1E_3+E_4$, and so on. Remarkably, we see that the local order of a random group has expected value $E_1$ and variance $M_2 - M_1^2 = E_2$.\smallskip

Since this is the first time the result is published, let me list some computations. As formal power series, we get expected value

$$M_1 = E_1 = q+2q^2+2q^3+3q^4+2q^5+4q^6+\ldots,$$\pagebreak[2]

variance

$$V = E_2 = q+3q^2+4q^3+7q^4+6q^5+12q^6+\ldots,$$

and higher moments
\begin{eqnarray*}
M_2 & = & q+4q^2+8q^3+15q^4+20q^5+32q^6+\ldots\\
M_3 & = & q+8q^2+26q^3+63q^4+116q^5+208q^6+\ldots,\\
M_4 & = & q+16q^2+80q^3+255q^4+608q^5+1280q^6+\ldots,
\end{eqnarray*}

and so on.

Finally, I give a table giving (approximatively) expected value $M_1$, variance $V$, and higher moments $M_2$, $M_3$ and $M_4$ of the local order for various primes $p$. Recall that all values are simply obtained from the power series by plugging in $q=\frac{1}{p}$:\bigskip

\begin{tabular}{|l||c|c|c|c|c|c|c|}
\noalign{\hrule}
& $p=2$ & $p=3$ & $p=5$ & $p=7$ & $p=11$& $p=13$ & $p=17$\\
\hline
$M_1$ & 1.6067 & 0.6822 & 0.3017 & 0.1909 & 0.1091 & 0.0898 & 0.0662\\
\noalign{\hrule}
$V$ & 2.7440 & 0.9494 & 0.3660 & 0.2191 & 0.1192 & 0.0968 & 0.0701\\
\noalign{\hrule}
$M_2$ & 5.3255 & 1.4148 & 0.4571 & 0.2556 & 0.1311 & 0.1048 & 0.0745\\
\noalign{\hrule}
$M_3$ &24.4734 & 3.9984 & 0.8848 & 0.4173 & 0.1817 & 0.1387 & 0.0926\\
\noalign{\hrule}
$M_4$ & 145.5087 & 14.7677 & 2.2088 & 0.8596 & 0.3053 & 0.2189 & 0.1340\\
\noalign{\hrule}
\end{tabular}\bigskip

Recall that the local order is the $p$-logarithm of the usual order, so the trivial group has local order $0$. This is why moments of less than $1$ are possible.

\subsection{Rank}

\begin{theorem}
The probability that a randomly chosen group has rank $r$ is 

$$P(\rk(G) = r) = \left(\prod_{i=1}^{\infty}(1-q^i)\right)\frac{q^{r^2}}{\left(\prod_{i=1}^r (1-q^i)\right)^2}.$$
\end{theorem}

This formula was already contained in Cohen and Lenstra's original paper \cite[Thm. 6.3]{CL83}, but was independently proven by Rudvalis and Shinoda \cite{RS88}. Later on, a new proof by means of the cycle index was given by Fulman \cite[Thm. 15]{Ful97}.

In fact, the theorems of Rudvalis and Shinoda look very different from the version given above. They make statements about the probability that a random matrix from $\GL(n,p)$ has a fixed space of dimension $r$. But it is easy to see (cf. \cite[Lemma 11]{Ful97}) that the dimension of the fixed space of a matrix $M \in \GL(n,p)$ equals the rank of $\underline{\lambda}'_{X-1}$, i.e., the number of parts of the partition corresponding to the polynomial $X-1$ in the Jordan-Chevalley normal form. Since for $n\rightarrow \infty$ the distribution of this partition is given by the Cohen-Lenstra probability, the above theorem is equivalent to the following corollary, and this is the form in which Rudvalis/Shinoda and Fulton have given their theorems:

\begin{corollary}\label{cor:fixedspace}
The probability that a randomly chosen matrix in $\GL(n,p)$ has a fixed space of dimension $r$ approaches, as $n\rightarrow \infty$, 

$$\left(\prod_{i=1}^{\infty}(1-p^{-i})\right)\frac{p^{-r^2}}{\left(\prod_{i=1}^r (1-p^{-i})\right)^2}.$$
\end{corollary}

Washington, who is clearly in the number theory fraction, published this as a remarkable observation \cite{Was86}, but he did not deduce the general theorem \ref{thm:conjugacy}. Also, no immediate reason for this coincidence is known (or for the general agreement between the Cohen-Lenstra probability and the probability of partitions appearing in the Jordan-Chevalley normal form), although this might be simply due to lack of research.

\subsubsection{Higher moments of the rank}

A closed formula for the higher moments of the rank of a random group is not known. However, if we consider the quantity $p^{\rk(G)}$ instead of $\rk(G)$, then more can be said. Cohen and Martinet \cite[(1.1)(d)]{CM87} give the following formula for its higher moments:

\begin{theorem}
The $k$-th moment of $p^{\rk(G)}$ is (with $q=\frac{1}{p}$)

$$\sum_{r\geq 0}p^{kr}\cdot P(\rk(G)=r) = \sum_{i=0}^{k}\left(q^{-i(k-i)}\frac{\prod_{j=1}^{k} (1-q^j)}{\left(\prod_{j=1}^{i} (1-q^j)\right)\left(\prod_{j=1}^{k-i} (1-q^j)\right)}\right).$$
\end{theorem}

The same formula was independently proven by Fulman \cite[Thm. 18,19]{Ful97}. He also pointed out that the summands may be interpreted as the $q$-analogue $S_q(k,i)$ of the Stirling numbers of second kind (cf. \cite{BDS94}).
%
%

\subsection{Rank and order combined}

\begin{theorem}
The probability that a finite abelian $p$-group has order $p^n$ and rank $r$ is

$$P\left(\atop{\ord(G) = n, }{\rk(G) = r}\right) = \left(\prod_{i=1}^{\infty}(1-q^i)\right)\frac{\displaystyle q^{n-r}\prod_{i=1}^{n-1} (1-q^i)}{\displaystyle |\GL(r,p)|\left(\prod_{i=1}^{r-1} (1-q^i)\right)\left(\prod_{i=1}^{n-r} (1-q^i)\right)}.$$
\end{theorem}

This theorem seems to be missing in the number theory community. It was proven by Fulman \cite[Thm. 16]{Ful97} using the cycle index.

\subsection{Exponent}\label{sect:expon}

\begin{theorem}
The probability that a random group has ($p$-adic) exponent at most $e$ is

$$P(\exp{G} \leq e) = \prod_{\atop{i=1}{i \equiv 0,\pm (e+1) \bmod (2e+3)} }^{\infty} (1-q^i),$$
\end{theorem}

where the index runs through all positive integers that satisfy one of the congruences.

This theorem was first proven by Cohen \cite{Coh85} and was independently rediscovered by Fulman \cite[Thm. \!\!21]{Ful97} via his Young Tableau Algorithm. A different and very simple proof is given in \cite{Len08} by means of so-called CL-maps.

All proof methods involve the generalized Ramanujan-Rogers identities \cite[Thm. 7.5]{And76}. The case $e=1$ occurred already in \cite{CL83} and involves the original Ramanujan-Rogers identity.

\subsection{$u$-probabilities}\label{sect:u-probs}

\begin{definition}\label{def:u-prob}
Let $u$ be a positive integer and $G$ a finite abelian $p$-group. The $u$-\emph{probability of $G$}, denoted by $P_u(G)$, is the probability that $G$ is obtained by the following random process: \index{u-probability@$u$-probability}\nomenclature[Pu]{$P_u$}{$u$-Cohen-Lenstra probability, cf. \ref{def:u-prob}}

\begin{enumerate}
\item Choose randomly a $p$-group $H$ with respect to the Cohen-Lenstra probability.
\item Choose $u$ elements $g_1,\ldots,g_u$ uniformly at random.
\item Output $H/\langle g_1,\ldots, g_u\rangle$.
\end{enumerate}

Here, $\langle g_1,\ldots, g_u\rangle$ denotes the subgroup generated by $g_1,\ldots ,g_u$.
\end{definition}

The $u$-probabilities are important for studying class groups of number fields (cf.\ \cite{CL83}, \cite{Mal06} or \cite{Len09} for details). They have extensively been studied by Cohen and Lenstra \cite{CL83} and others. By means of $\zeta$-functions, Cohen and Lenstra derived the following explicit formula:

\begin{theorem}\label{thm:uprobs}
Let $u>0$ be an integer, and let $G$ be a finite abelian $p$-group of order $n$. Then 

\begin{eqnarray*}
P_u(G) & = & \frac{1}{n^u\prod_{i=1}^{u}(1-p^{-i})}P(G)\\
& = & n^{-u}\frac{1}{\#\Aut(G)}\prod_{i=u+1}^{\infty}(1-p^{-i}).
\end{eqnarray*}
\end{theorem}

\begin{proof}
\cite[Example 5.9]{CL83}
\end{proof}

In the same paper, you can find explicit formulas for the $u$-probability that a $p$-group is of a certain order or certain rank, is cyclic, is elementary, and formulas for the expected values of the size of a group and the number of elements with given annihilator \cite[examples 5.8--5.13, theorem 6.3]{CL83}. A formula for the $u$-probability of the exponent of a $p$-group is given in \cite{Coh85}.



%
%
%
%

\bibliographystyle{plain}
\bibliography{Cohen-Lenstra-Heuristic}

\end{document}